\documentclass[11pt,a4paper]{amsart}
\usepackage[dvips]{graphicx}
\usepackage{amsmath}
\usepackage{amscd}
\usepackage{amssymb}
\usepackage{latexsym}
\usepackage[T1]{fontenc}
\usepackage[latin1]{inputenc}
\usepackage[all]{xy}

\newtheorem{Prop}{Proposition}[section]
\newtheorem{Def}[Prop]{Definition}
\newtheorem{Theorem}[Prop]{Theorem}
\newtheorem{Lemma}[Prop]{Lemma}

\newtheoremstyle{Example}{\topsep}{\topsep}%
  {}%         Body font
  {}%         Indent amount (empty = no indent, \parindent = para indent)
  {\bfseries}% Thm head font
  {.}%        Punctuation after thm head
  { }%     Space after thm head (\newline = linebreak)
  {\thmname{#1}\thmnumber{ #2}\thmnote{ #3}}%         Thm head spec

\theoremstyle{Example}
\newtheorem{Example}[Prop]{Example}
\newtheoremstyle{Observation}{\topsep}{\topsep}%
  {}%         Body font
  {}%         Indent amount (empty = no indent, \parindent = para indent)
  {\bfseries}% Thm head font
  {.}%        Punctuation after thm head
  { }%     Space after thm head (\newline = linebreak)
  {\thmname{#1}\thmnumber{ #2}\thmnote{ #3}}%         Thm head spec

\theoremstyle{Example}
\newtheorem{Remark}[Prop]{Remark}
\newtheoremstyle{Observation}{\topsep}{\topsep}%
  {}%         Body font
  {}%         Indent amount (empty = no indent, \parindent = para indent)
  {\bfseries}% Thm head font
  {.}%        Punctuation after thm head
  { }%     Space after thm head (\newline = linebreak)
  {\thmname{#1}\thmnumber{ #2}\thmnote{ #3}}%         Thm head spec

\theoremstyle{Observation}

\title[The many polarizations of powers of maximal ideals]{The many polarizations of powers of maximal ideals}
\author{Henning Lohne}
\address{Matematisk Institutt\\
         Johs. Brunsgt. 12\\
         5008 Bergen}
\email{henning.lohne@math.uib.no}
\date{\today}
\begin{document}
\begin{abstract}
In this paper, we study different polarizations of powers of the maximal ideal, and polarizations of their related square-free versions. For $n=3$, we show that every minimal free cellular resolution of $m^d$ comes from a certain polarization of the ideal $m^d$. This result is not true for $n=4$. When $I$ is a square-free ideal, we show that the Alexander dual of any polarization of $I$ is a polarization of the Alexander dual ideal of $I$. We apply this theorem, and study different polarizations of the ideals $m_\mathrm{sq.fr}^d$ and their Alexander duals $m_\mathrm{sq.fr}^{n-d+1}$ simultaneously, by giving a combinatorial description corresponding to such polarizations, with a natural dualization. We apply this theory, and study the case of $d=2$ and $d=n-1$ in more detail. Here, we show that there is a one-to-one correspondence between spanning trees of $K_n$ and the maximal polarizations of these ideals.
\end{abstract}

\bibliographystyle{plain}

\maketitle

\section{Introduction}

Let $S=k[x_1,\dots,x_n]$ be the polynomial ring in $n$ variables over a field $k$. Sinefakopoulos \cite{Si}, and Nagel and Reiner \cite{NR} describes a nice way for giving a minimal cellular resolution of $m^d$. The resolution they describe comes from a polarization of the ideal, which we will call the box polarization of $m^d$, or equivalently the box polarization of the ideal $I_d=(x_1,\dots,x_{n'})^d_\mathrm{sq.fr.}$ consisting of all square-free monomials of degree $d$ in the polynomial ring $S'=k[x_1,\dots,x_{n'}]$, where $n'=n+d-1$. We present here the definition of what we mean by a polarization of an ideal, and some basic facts about them. See \cite{Ya} for more details.
\begin{Def}
Let $I$ be an ideal in $S$. A polarization of $I$ is defined as a square-free monomial ideal $\widetilde{I}$ in
$$\widetilde{S}:=k\left[x_1^{(1)},\dots,x_1^{(r_1)},x_2^{(1)},\dots,x_2^{(r_2)},\dots,x_n^{(r_n)}\right]$$
such that the sequence
$$\sigma=\left(x_1^{(1)}-x_1^{(2)}, x_1^{(1)}-x_1^{(3)},\dots, x_1^{(1)}-x_1^{(r_1)},x_2^{(1)}-x_2^{(2)},\dots, x_n^{(1)}-x_n^{(r_n)}\right)$$
is a regular $\widetilde{S}/\widetilde{I}$-sequence, and that $\widetilde{I}\otimes \widetilde{S}/(\sigma)\cong I$. The corresponding homomorphism $\widetilde{I}\rightarrow I$ is called the depolarization of $\widetilde{I}$. An ideal $I$ is said to be maximal polarized if there exists no non-trivial polarization of $I$.
\end{Def}

By this definition, it is clear that the $\mathbf{Z}$-graded Betti numbers of $I$ and $\widetilde{I}$ are the same, and a minimal (cellular) free resolution for $\widetilde{I}$ gives rise to a minimal (cellular) free resolution of $I$.

The box polarization of $m^d$ is the ideal
$$B_{nd} = \left(x_{i_1}^{(1)}x_{i_2}^{(2)}\cdots x_{i_d}^{(d)}\,\middle|\,1\le i_1\le i_2 \le \cdots \le i_d \le n\right).$$

We are interested in studying different polarization of the ideals $m^d$ and $I_d$ in more detail.

In Section \ref{xyz} we show that a minimal free cellular resolution of $(x,y,z)^d$ corresponds to a maximal polarization of the ideal. However this is not the case when $n>3$, since it is possible to show that the Eliahou--Kervaire resolution of $(x_1,x_2,x_3,x_4)^2$, which is known to be cellular (i.e. \cite[Fact 5.2]{batzies2002discrete} or \cite[Theorem 3.4]{peeva2008minimal}), can not occur from any polarization of the ideal.

In Section \ref{sqfr}, we study polarizations of square-free monomial ideals and their Alexander dual. We show in fact that if $\widetilde{I}$ is a polarization of an ideal $I$, then the Alexander dual $D(\widetilde{I})$ is a polarization of the Alexander dual ideal $D(I)$ of $I$. We use this result to study how polarization of the ideals $I_d$ and their Alexander duals $I_{n-d+1}$ are related. We describe these polarizations by partitioning some sets, such that the dual partitioning corresponds to the polarization of the Alexander dual ideal. We also give a criterion for when a partition gives a polarization. We give two natural examples of maximal polarizations which are in some sense self dual.

Finally, in Section \ref{sqfr2}, we examine the special case $d=2$. We will show that there exists a nice classification of maximal polarizations analogous to the results in \cite{Fl}, by analyzing the possible cellular resolutions of the Alexander dual ideal $I_{n-1} = D(I_2)$ instead. We also discuss non-maximal polarization, and how they correspond to edge ideals. We will give a criterion for when there exist polarizations of such edge ideals.

We briefly recall some basic definition. We write $[n]=\{1,2,\dots,n\}$, and a subset $F\subseteq [n]$ is called a face. A simplicial complex is a collection of faces $\Delta$, such that if $F\in \Delta$ and $G\subseteq F$, then $G\in \Delta$. The Stanley--Reisner ideal of the simplicial complex $\Delta$ is the square-free monomial ideal $I_\Delta=\left(x^\sigma\,|\, \sigma\not\in \Delta \right)$ generated by monomials corresponding to non-faces $\sigma$ of $\Delta$. The Stanley--Reisner ring of $\Delta$ is the quotient ring $S/{I_{\Delta}}$. Simplicial complexes comes with a reduced chain complex mapping faces to codimension $1$ faces. More generally, if we already have a chain complex, for instance a minimal free resolution, it would be nice to find a simplicial complex such that our resolution more or less is the reduced chain complex of this. However, this rarely happens, but it can often be done if we introduce the more general polyhedral cell complexes as we will see. This definition and examples can be found in \cite[Chapter 4]{MS}.

\begin{Def}
A polyhedral cell complex $X$ is a finite collection of convex polytopes, called the faces of $X$, satisfying the following two properties:
\begin{itemize}
\item[1.]
If $P$ is a polytope in $X$ and $F$ is a face in $P$, then $F$ is in $X$.
\item[2.]
If $P$ and $Q$ are in $X$, then $P\cap Q$ is a face in both $P$ and $Q$.
\end{itemize}
\end{Def}

A polyhedral cell complex also comes with a reduced chain complex
$$k^{\# F_{-1}} \leftarrow k^{\# F_0} \leftarrow k^{\# F_1} \leftarrow \cdots \leftarrow k^{\# F_d},$$
with basis given by the faces and differential
$$\partial(F)=\sum_{G\text{ facet in } F} \text{sign}(G,F)\cdot F,$$
where sign is determined by an (arbitrarily) orienting of the faces, with $\text{sign}(G,F)=1$ if the orientation on $F$ induces the orientation on $G$, and $-1$ if not.

\begin{Def}
$X$ is said to be a labeled cell complex if its $r$ vertices are labeled by vectors $\mathbf{a}_1,\dots, \mathbf{a}_r$ in $\mathbf{N}^n$ and the faces $F$ are labeled by $\mathbf{a}_F\in \mathbf{N}^n$, where $\mathrm{lcm}(\mathbf{x}^{\mathbf{a}_i}\,|\, i\in F)$. We label the empty face $\emptyset$ with $\mathbf{0}$.

We have the cellular free complex $\mathcal{F}_X$, supported on $X$ given as
$$\mathcal{F}_X=\bigoplus_{F\in X} S(-\mathbf{a}_F)$$
and differential
$$\partial(F)=\sum_{G \text{ facet in } F} \mathrm{sign}(G,F)\mathbf{x}^{\mathbf{a}_F-\mathbf{a}_G} G.$$
$F$ and $G$ are considered both as faces and as basis elements in degree $\mathbf{a}_F$ and $\mathbf{a}_G$. $\mathcal{F}_X$ is a cellular resolution if it is acyclic (homology only in degree $0$).
\end{Def}

\section{Polarizations of $(x,y,z)^d$}\label{xyz}

In the article of Nagel and Reiner \cite{NR}, the authors are interested in producing a minimal free cellular resolution of the ideal $m^d$, and restrict this to a Borel fixed ideal $I\subseteq m^d$, to get a minimal free cellular resolution of $I$. To do this, they introduce what they call the complex of boxes resolution, which they get from a polarization of the ideal $m^d$ (and $m_{\mathrm{sq.fr}}^d)$. This polarization is what we call the box polarization. For the case $n=3$, we know that there are several other possible minimal free cellular resolutions of $m^d$. We want to show that every such minimal free cellular resolution can be obtained by using a suitable polarization of the ideal $m^d$.

The generators of the ideal $(x,y,z)^d$ and the linear relations between them can be arranged in a triangular shaped graph as shown in the figure below for $d=4$.

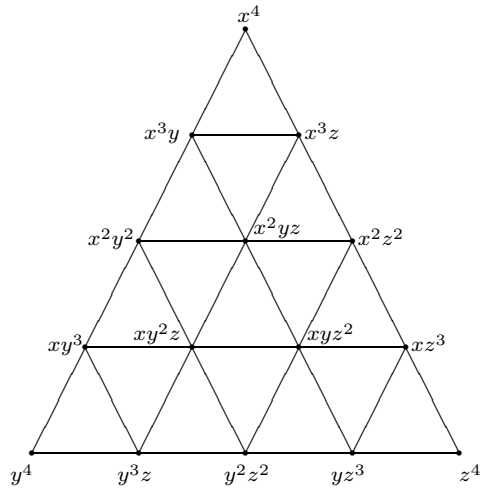
\newcommand{\trekantA}{\vskip 1cm
\begin{picture}(200,160)
\put(2,2){$_{y^4}$}
\put(10,10){\circle*{2}}
\put(10,10){\line(1,2){80}}

\put(50,10){\line(1,2){60}}
\put(50,10){\circle*{2}}
\put(90,10){\line(1,2){40}}
\put(90,10){\circle*{2}}
\put(130,10){\line(1,2){20}}
\put(130,10){\circle*{2}}

\put(87,175){$_{x^4}$}
\put(90,170){\circle*{2}}
\put(90,170){\line(1,-2){80}}

\put(52,131){$_{x^3y}$}
\put(70,130){\line(1,-2){60}}
\put(70,130){\circle*{2}}
\put(31,91){$_{x^2y^2}$}
\put(50,90){\line(1,-2){40}}
\put(50,90){\circle*{2}}
\put(16,51){$_{xy^3}$}
\put(30,50){\line(1,-2){20}}
\put(30,50){\circle*{2}}

\put(170,2){$_{z^4}$}
\put(170,10){\circle*{2}}
\put(10,10){\line(1,0){160}}

\put(30,50){\line(1,0){120}}
\put(50,90){\line(1,0){80}}
\put(70,130){\line(1,0){40}}

\put(110,130){\circle*{2}}
\put(130,90){\circle*{2}}
\put(150,50){\circle*{2}}

\put(112,131){$_{x^3z}$}
\put(132,91){$_{x^2z^2}$}
\put(152,51){$_{xz^3}$}

\put(70,50){\circle*{2}}
\put(90,90){\circle*{2}}
\put(110,50){\circle*{2}}

\put(48,55){$_{xy^2z}$}
\put(93,95){$_{x^2yz}$}
\put(113,55){$_{xyz^2}$}

\put(42,2){$_{y^3z}$}
\put(82,2){$_{y^2z^2}$}
\put(122,2){$_{yz^3}$}

%\put(60,5){\mbox{Figure 3.1.}}
\end{picture}}

\newcommand{\trekantB}{\vskip 1cm
\begin{picture}(200,160)
\put(25,25){$_{y^3}$}
\put(10,10){\circle*{2}}
\put(10,10){\line(1,2){80}}

\put(50,10){\line(1,2){60}}
\put(50,10){\circle*{2}}
\put(90,10){\line(1,2){40}}
\put(90,10){\circle*{2}}
\put(130,10){\line(1,2){20}}
\put(130,10){\circle*{2}}

\put(87,145){$_{x^3}$}
\put(90,170){\circle*{2}}
\put(90,170){\line(1,-2){80}}

\put(63,103){$_{x^2y}$}
\put(70,130){\line(1,-2){60}}
\put(70,130){\circle*{2}}
\put(65,70){$_{xy}$}
\put(50,90){\line(1,-2){40}}
\put(50,90){\circle*{2}}
\put(43,67){$_{xy^2}$}
\put(30,50){\line(1,-2){20}}
\put(30,50){\circle*{2}}

\put(145,25){$_{z^3}$}
\put(170,10){\circle*{2}}
\put(10,10){\line(1,0){160}}

\put(30,50){\line(1,0){120}}
\put(50,90){\line(1,0){80}}
\put(70,130){\line(1,0){40}}

\put(110,130){\circle*{2}}
\put(130,90){\circle*{2}}
\put(150,50){\circle*{2}}

\put(87,113){$_{x^2}$}
\put(123,67){$_{xz^2}$}

\put(70,50){\circle*{2}}
\put(90,90){\circle*{2}}
\put(110,50){\circle*{2}}

\put(83,67){$_{xyz}$}
\put(105,103){$_{x^2z}$}
\put(105,70){$_{xz}$}

\put(47,32){$_{y^2}$}
\put(85,32){$_{yz}$}
\put(126,32){$_{z^2}$}
\put(63,25){$_{y^2z}$}
\put(103,25){$_{yz^2}$}

%\put(60,5){\mbox{Figure 3.1.}}
\end{picture}}

\begin{center}
\trekantA
\end{center}

The triangles corresponding to three vertices all adjacent to each other can be labeled by the greatest common divisor of its vertices. We see in the figure below that these correspond to a set of up triangles labeled by the generators of $m^{d-1}$, and a set of down triangles labeled by the generators of $m^{d-2}$.

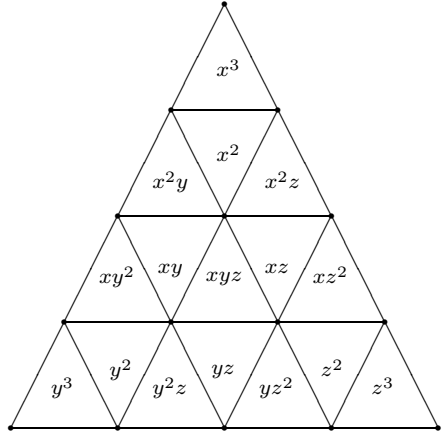
\begin{center}
\trekantB
\end{center}

We denote the simplicial complex consisting of the vertices and edges from the triangular graph and the $2$-cells corresponding to the up triangles and down triangles above by $\Gamma$. This simplicial complex gives rise to a cellular resolution of $m^d$ by labeling the faces by the least common multiple of the monomials corresponding to the vertices. This resolution is not minimal, but we have the following result.

\begin{Theorem}
For every choice of removing exactly one edge from each down triangle in $\Gamma$, there exists a polarization $I$ of $m^d$ such that the corresponding polyhedral cell complex $\Delta$ supports a minimal cellular resolution of $I$ (and therefore also for $m^d$). Furthermore, every minimal cellular resolution of $m^d$ comes from such a polarization.
\end{Theorem}

\begin{proof} We show first that every minimal free cellular resolution of $m^d$ is indeed on this form. This is because we know that the minimal resolution of $m^d$ is linear, so a minimal free cellular resolution of $m^d$ must contain a subgraph of the $1$-skeleton of $\Gamma$. A down triangle labeled $n$ from $m^{d-2}$ consists of the vertices $m_1=nxy, m_2=nxz$ and $m_3=nyz$. We observe that $\mathrm{lcm}(m_1,m_2)=\mathrm{lcm}(m_1,m_3)=\mathrm{lcm}(m_2,m_3)=\mathrm{lcm}(m_1,m_2,m_3)$. Suppose that $\mathcal{F}$ is a minimal free cellular resolution of $m^d$. We must then have that $\mathcal{F}_{\leq \mathbf{d}}$ is acyclic for every multidegree $\mathbf{d}\in \mathbf{Z}^3$ (\cite[Prop. 4.5]{MS}). Letting $\mathbf{d}=\mathrm{deg}(\mathrm{lcm}(m_1,m_2,m_3))$, we see that $\mathcal{F}_{\leq \mathbf{d}}$ consists only of the three vertices from the down triangle, hence it must either contain all edges from the triangle and the $2$-face, or it must contain exactly two edges. But since $\mathrm{lcm}(m_1,m_2)=\mathrm{lcm}(m_1,m_2,m_3)$, the first case would not give a minimal resolution. Hence we must have exactly two edges from each down triangle.

Next, we let $F$ be a $2$-cell labeled by the monomial $m$ of multidegree $\mathbf{d}$. We want to look at the complex $\mathcal{F}_{\leq \mathbf{d}}$. Since the minimal free resolution is linear, we must have that $\mathrm{deg}(m) = d+2$, which means that $\mathcal{F}_{\leq \mathbf{d}}$ is at most supported on vertices of the form:

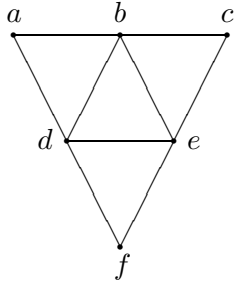
\begin{center}
\begin{picture}(100,100)
\put(7.5,95){$a$}
\put(10,90){\circle*{2}}

\put(47.5,95){$b$}
\put(50,90){\circle*{2}}

\put(87.5,95){$c$}
\put(90,90){\circle*{2}}

\put(19,47){$d$}
\put(30,50){\circle*{2}}

\put(75,47){$e$}
\put(70,50){\circle*{2}}

\put(47.5,0){$f$}
\put(50,10){\circle*{2}}

\put(10,90){\line(1,0){80}}
\put(30,50){\line(1,0){40}}
\put(30,50){\line(1,2){20}}
\put(70,50){\line(-1,2){20}}
\put(50,10){\line(1,2){40}}
\put(50,10){\line(-1,2){40}}

\end{picture}
\end{center}
So $F$ is supported on a subset of these vertices. We want to show that if $F$ is a supported on a subset of vertices of this form, then the $1$-skeleton of $F$ is a cycle without chordes. So assume that $a$ is in $F$. Then clearly $ab$ and $ad$ are faces of $F$. But then $bd$ can not be a face of $F$, since if $m'$ is the least common multiplum of the monomials corresponding to the down triangle $abd$, and if $\mathbf{d'}$ is the degree of $m'$, then $\mathcal{F}_{\le\mathbf{d'}}$ is just the complex restricted to the down triangle $abd$, which is impossible as we just have shown above. Similar arguments shows that we can not have the vertex $c$ and the chord $be$, and we can not have the vertex $f$ and the chord $de$. 

The proof will now follow from the construction of the polarization for such a configuration which we give in Subsection \ref{construction} below.
\end{proof} 
\subsection{The construction}\label{construction}

Let $M_d$ denote the set of minimal generators of the ideal $m^d$. So the set of down triangles are in one-to-one correspondence to $M_{d-2}$ by taking the greatest common divisor of the monomials in its vertices. Now suppose we remove exactly one edge from each of the down triangles as above, and we denote the correspondig polyhedral cell complex by $\Delta$. That is, the $0$-cells and $1$-cells of $\Delta$ are the graph obtained by removing one edge from each down triangle, and the $2$-cells are all the internal regions of this planar graph. Let $m\in M_{d-2}$ correspond to a down triangle $mxy,mxz,myz$. If the edge removed from this triangle consists of the vertices $mxy$ and $mxz$, then $m$ is called an $x$-triangle, if the edge removed consists of the vertices $mxy$ and $myz$ it is called an $y$-triangle, and finally if it consists of $mxz$ and $myz$ it is called a $z$-triangle.

\begin{center}
\begin{picture}(210,70)
\put(22,25){\line(1,2){20}}
\put(22,25){\line(-1,2){20}}
\put(22,25){\circle*{2}}
\put(42,65){\circle*{2}}
\put(2,65){\circle*{2}}
\put(0,5){$x$-triangle}

\put(102,25){\line(1,2){20}}
\put(82,65){\line(1,0){40}}
\put(102,25){\circle*{2}}
\put(122,65){\circle*{2}}
\put(82,65){\circle*{2}}
\put(80,5){$y$-triangle}

\put(182,25){\line(-1,2){20}}
\put(162,65){\line(1,0){40}}
\put(182,25){\circle*{2}}
\put(202,65){\circle*{2}}
\put(162,65){\circle*{2}}
\put(160,5){$z$-triangle}

\end{picture}
\end{center}

We can now describe how we construct the polarization corresponding to this polyhedral cell complex, which gives rise to a minimal free cellular resolution of $m^d$.

First of all, we polarize $z^{d-1}x$ to $z^{d-1}x_1$. Then we assume that we have polarized the $x$-variables in the monomials from $z^{d-i}M_{i}(x,y)$ in the variables $x_1,\dots, x_i$, and that this polarization corresponds to a maximal chain $\mathbf{s}(i)$ of faces $\emptyset=s_{i,0} \subset s_{i,1}\subset \cdots \subset s_{i,i-1} \subset s_{i,i} = [i]$, such that $x^{i-j}$ in $x^{i-j}y^jz^{d-i}$ is polarized as $\prod_{\underset{k\in [i],}{k \not\in s_{i,j}}}x_k.$ \\

\begin{Example}
To make things clearer, we illustrate what we mean by an example. If $i=3$, and if the monomials $x^3z^{d-3},x^2yz^{d-3},xy^2z^{d-3},y^3z^{d-3}$ are already polarized to for instance $x_1x_2x_3z^{d-3}, x_1x_3yz^{d-3}, x_3y^2z^{d-3}, y^3z^{d-3}$. Then this polarization corresponds to the sequence $\mathbf{s}(3)$ which is $\emptyset \subset \{2\} \subset \{1,2\} \subset \{1,2,3\}$. 
\end{Example}
\vspace{0.2cm}

We now continue and construct the sequence $\mathbf{s}(i+1)$ iteratively as follows. First we construct a sequence $\mathbf{s}(i)'$ by letting $s_{i,j}'=s_{i,j}\cup\{i+1\}$. Next, we let $s_{i+1,0}=\emptyset$. For $j=0,1,\dots, i$, if $j<i$ and if $x^{i-1-j}y^jz^{d-i-1}$ is an $x$-triangle we let $s_{i+1,j+1}=s_{i+1,j}\cup (s_{i,j+1}'\setminus s_{i,j}')$. Otherwise we let $s_{i+1,j+1}=s_{i,j}'$. We now polarize the monomials in $z^{d-i-1}M_{i+1}(x,y)$ by the sequence $\mathbf{s}(i+1)$.

The polarization of the $y$-variables and the $z$-variables are done in exactly the same way and the details are obmitted. \\

\begin{Example}
Consider the following example:

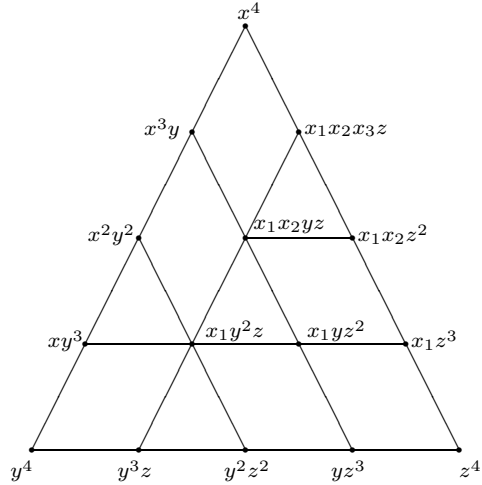
\begin{center}
\begin{picture}(200,200)
\put(2,2){$_{y^4}$}
\put(10,10){\circle*{2}}
\put(10,10){\line(1,2){80}}

\put(50,10){\line(1,2){60}}
\put(50,10){\circle*{2}}
%\put(90,10){\line(1,2){40}}
\put(90,10){\circle*{2}}
%\put(130,10){\line(1,2){20}}
\put(130,10){\circle*{2}}

\put(87,175){$_{x^4}$}
\put(90,170){\circle*{2}}
\put(90,170){\line(1,-2){80}}

\put(52,131){$_{x^3y}$}
\put(70,130){\line(1,-2){60}}
\put(70,130){\circle*{2}}
\put(31,91){$_{x^2y^2}$}
\put(50,90){\line(1,-2){40}}
\put(50,90){\circle*{2}}
\put(16,51){$_{xy^3}$}
%\put(30,50){\line(1,-2){20}}
\put(30,50){\circle*{2}}

\put(170,2){$_{z^4}$}
\put(170,10){\circle*{2}}
\put(10,10){\line(1,0){160}}

\put(30,50){\line(1,0){120}}
\put(90,90){\line(1,0){40}}
%\put(70,130){\line(1,0){40}}

\put(110,130){\circle*{2}}
\put(130,90){\circle*{2}}
\put(150,50){\circle*{2}}

\put(112,131){$_{x_1x_2x_3z}$}
\put(132,91){$_{x_1x_2z^2}$}
\put(152,51){$_{x_1z^3}$}

\put(70,50){\circle*{2}}
\put(90,90){\circle*{2}}
\put(110,50){\circle*{2}}

\put(75,55){$_{x_1y^2z}$}
\put(93,95){$_{x_1x_2yz}$}
\put(113,55){$_{x_1yz^2}$}

\put(42,2){$_{y^3z}$}
\put(82,2){$_{y^2z^2}$}
\put(122,2){$_{yz^3}$}

%\put(60,5){\mbox{Figure 3.1.}}
\end{picture}
%}
\end{center}
Here we have polarized the $x$-variables from $z^4M_0(x,y)$ to $z^1M_3(x,y)$ and we now want to polarize the $x$-variables in the monomials of $z^0 M_4(x,y)$. The sequence $\mathbf{s}(3)$ corresponding to the polarization in $z^1M_3(x,y)$ is the sequence $\emptyset \subset \{3\} \subset \{2,3\} \subset \{1,2,3\}$. We construct the sequence $\mathbf{s}'(3)$ as $\{4\} \subset \{3,4\} \subset \{2,3,4\} \subset \{1,2,3,4\}$, and the sequence $\mathbf{s}(4)$ as follows: Let $s_{4,0}=\emptyset$. Next, since the first down triangle, i.e. the triangle labeled $x^2$ corresponding to $j=0$, is an $x$-triangle we have that $s_{4,1}=s_{4,0}\cup (s_{3,1}'\setminus s_{3,0}')=\{3\}$. The second down triangle labeled $xy$ is also an $x$-triangle. We therefore have that $s_{4,2}=s_{4,1}\cup (s_{3,2}'\setminus s_{3,1}') = \{2,3\}$. The third and last down triangle labeled $y^2$ is not an $x$-triangle. We therefore have that $s_{4,3}=s_{3,2}'=\{2,3,4\}$, and finally also that $s_{4,4}=\{1,2,3,4\}$. This means that $x^4$ is polarized to $x_1x_2x_3x_4$, $x^3y$ to $x_1x_2x_4$, $x^2y^2$ to $x_1x_4y^2$ and $xy^3$ to $x_1y^3$.
\end{Example}

So for every such polyhedral cell complex $\Delta$, we create this ideal $I$ and claim it is a polarization of $m^d$, and that it has a minimal free resolution supported on $\Delta$. Since the generators of $I$ depolarizes to the generators of $m^d$, it will follow that $I$ is a polarization of $m^d$ if we can show that $I$ has a minimal free resolution supported on $\Delta$. This is because this resolution depolarizes to a minimal free resolution of $m^d$. To show that $I$ has a minimal free resolution supported on $\Delta$, we need three technical lemmas.

\begin{Lemma}\label{Observe}
Let $m=x^ay^bz^c$ and $n=x^{a'}y^{b'}z^c$ be two generators of $m^d$ with $a<a'$. If $m_x$ and $n_x$ denotes the corresponding polarizations of $x^a$ and $x^{a'}$ in $I$, then $m_x|n_x$. Furthermore, if $m'=x^ay^bz^c$ and $n'=x^{a'}y^bz^{c'}$ are two generators of $m^d$ with $a<a'$, and if $m'_x$ and $n'_x$ denotes the corresponding polarizations of $x^a$ and $x^{a'}$ in $I$, then $m'_x|n'_x$. Since the construction of the polarization of $y$ and $z$ is done precisely as for $x$, similar results also holds for the $y$ and $z$ part.
\end{Lemma}
\begin{proof}
In the first case, the $x$-variables are polarized according to the sequence $s_{i,0}\subseteq s_{i,1}\subseteq \cdots \subseteq s_{i,{i-1}} \subseteq s_{i,i}$, where $i = d-c$ so the result clearly holds by this construction. In the second case, the $x$-variables are polarized according to the sequence $s_{i,i},s_{i+1,i},s_{i+2,i},\dots,s_{d,i}$, where $i=b$, and the result follows if we can show that $s_{i+1,j}\cap [i] \subseteq s_{i,j}$. So suppose that we have shown that $s_{i+1,k}\cap [i] \subseteq s_{i,k}$. Then either $s_{i+1,k+1} = s_{i,k}'$ and $s_{i+1,k+1}\cap[i]=s_{i,k}\subseteq s_{i,k+1}$, or $s_{i+1,k+1}=s_{i+1,k}'\cup (s_{i,k+1}'\setminus s_{i,k}')$ and $s_{i+1,k+1}\cap[i] = s_{i+1,k}\cup (s_{i,k+1}\setminus s_{i,k})$, but $s_{i+1,k}\subseteq s_{i,k}\subseteq s_{i,k+1}$ and $(s_{i,k+1}\setminus s_{i,k})\subseteq s_{i,k+1}$ so $s_{i+1,k+1}\cap[i]\subseteq s_{i,k+1}$. For any $i$, we have that $s_{i,0}=\emptyset$, so by induction, the result holds for all $i$ and $j$.

Similar arguments on the polarization of $y$ and $z$ completes the proof.
\end{proof}

\begin{Lemma}\label{lemmalinearerelasjoner}
If $I$ is the polarization given above corresponding to the polyhedral cell complex $\Delta$, then there is a one-to-one correspondence between the edges (i.e. $1$-cells) of $\Delta$ and the linear relations between the generators of $I$. 
\end{Lemma}

\begin{proof}
First, we verify that there can not be a linear relation between two monomials where an edge is removed. Without loss of generality, we may assume that the edge corresponds to an $x$-triangle, and the two vertices are labeled $x^jy^{i}z^{d-i-j}$ and $x^jy^{i+1}z^{d-i-j-1}$. But then by the construction, $x^j$ is polarized different in the two monomials and since the monomials have different degrees in $y$ and $z$ it also have to have different polarizations in these variables. That means that there can not be a linear relation between them in $I$.

Next, we verify that there are linear relations between the generators connected by an edge of the outer boundary of $\Delta$ (i.e. the edges not in any down triangle). Again, we may assume that we are on the boundary containing the generators in $M_d(x,z)$ or $M_d(x,y)$. But again, by the construction, the polarization of the $x$-variables correspond either to the sequence $x_1,x_1x_2,\dots,x_1\cdots x_d$ or a sequence $\mathbf{s}(d)$, but in either case two consecutive monomials only differ by one $x$-variable. Applying the same argument to the $y$ or $z$-variable shows that two consecutive monomials also only differ by one $y$ or $z$-variable. Hence we will have a linear relation between them.

Finally, we will have to verify that an inner edge also corresponds to a linear relation between the monomials of its vertices. Again, we can assume without loss of generality that the vertices are labeled by polarizations of $x^jy^iz^{d-i-j}$ and $x^jy^{i+1}z^{d-i-j-1}$ and that the corresponding down triangle is not an $x$-triangle. So $x^j$ is polarized to the same product of $x$-variables in both monomials, and because of Lemma \ref{Observe}, we can apply the same argument as above for the remaining polarization of the $y$ and $z$-variables, and it follows that there is a linear relation between the two monomials in $I$.
\end{proof}

\begin{Lemma}
If $I$ and $\Delta$ are as above, and if $\mathcal{F}$ is the free cellular complex of $I$ supported on $\Delta$, then for any two generators $m$ and $n$ of $I$, the complex $\mathcal{F}_{\le \mathrm{lcm}(m,n)}$ is acyclic. Hence, $\mathcal{F}$ is a minimal free resolution of $I$. 
\end{Lemma}

\begin{proof}
By the construction of $\Delta$, it is easy to see that the complex $\mathcal{F}_{\le \mathbf{d}}$ can only have homology if it contains two vertices not connected by a path. That is, it can only have homology in degree $0$. The result will follow if we can show that any two generators are connected by a path of linear relations.

Suppose that $m_1$ is a generator of $I$ that depolarizes to $x^iy^jz^k$ and that $m_2$ is a generator that depolarizes to $x^{i'}y^{j'}z^{k'}$ and suppose that $m_1$ and $m_2$ divides $\mathrm{lcm}(m,n)$, i.e. they correspond to vertices in $\mathcal{F}_{\le \mathrm{lcm}(m,n)}$. We will show that there is a path of linear relations between $m_1$ and $m_2$ by showing that there is at least one other generator of $I$ in the same down triangle as $m_1$, closer $m_2$ (or visa versa), which also divides $\mathrm{lcm}(m_1,m_2)$ (and therefore also $\mathrm{lcm}(m,n))$. By iterating this process the result will follow because of Lemma \ref{lemmalinearerelasjoner}. We have the following possibilities:

\begin{itemize}
\item[1.] $i>i'$, $j>j'$ and $k<k'$. In this case we claim that the polarization of either $x^{i-1}y^jz^{k+1}$, which we call $g_1$, or $x^iy^{j-1}z^{k+1}$, which we call $g_2$, divides $\mathrm{lcm}(m_1,m_2)$. This is because $(m_1,g_1,g_2)$ forms a down triangle which either is not an $x$-triangle, in which case $g_2$ will be polarized by the same $x$-variables as $m_1$, a subset of the $y$-variables from $m_1$ and a subset of the $z$-variables of $m_2$. And likewise if the down triangle is not a $y$-triangle, then $g_1$ is polarized by a subset of the $x$-variables from $m_1$, the same $y$-variables and a subset of the $z$-variables of $m_2$. The reason $g_1$ and $g_2$ are polarized by a subset of the $z$-variables that occur in the polarization of $m_2$ (and similar for the claim on the $x$ and $y$-variables) is because of Lemma \ref{Observe}. So if we fix the degree of $x$ or $y$ then two consecutive generators are polarized in the $z$-variables according to a subset $\sigma \subset \sigma'$ which means that the variables in the first polarization is a subset of the variables in the second. By the assumption, it is possible to first fix $i$ and follow consecutive generators untill we reach the generator $x^iy^{j'}z^{k''}$, and then fix $j'$ and follow consecutive generators untill we reach $m_2$.

\item[2.] $i=i'$, $j>j'$ and $k<k'$. In this case we also claim that $g_1$ or $g_2$, as in case $1.$ will divide $\mathrm{lcm}(m_1,m_2)$. If the down triangle $(m_1,g_1,g_2)$ is not an $x$-triangle, it follows from the same reasons as above. So suppose that it is an $x$-triangle. It now follows by the same argument as above that $g_2$ is polarized by the same $y$-variables as $m_1$ and a subset of the $x$-variables of $m_1$. But since the down triangle is an $x$-triangle, we know that $g_2$ and $g_1$ are polarized by the same $z$-variables, and $g_1$ is polarized by a subset of the $z$-variables of $m_2$ by the same reason as above. It therefore follows that $g_2$ divides $\mathrm{lcm}(m_1,m_2)$.

\item[3.] Any case similar as above, but with possible $m_1$ and $m_2$ or some variables shifted. Then a similar argument as above can be used to find two other similar generators $g_1'$ and $g_2'$ also in $\mathrm{lcm}(m_1,m_2)$.

\end{itemize}

It is clear that this process will lead to a path of linear relations between the two generators. For instance, say that the distance between two generators is given as $|i-i'|+|j-j|+|k-k'|$. Then the first case will obviously find a generator in $\mathrm{lcm}(m_1,m_2)$ closer to $m_2$. In the second case, we will either find a generator in $\mathrm{lcm}(m_1,m_2)$ closer to $m_2$ or we will find a generator with similar distance to $m_2$, but a generator that will give rise to a situation of case $1$.

This means that in $\mathcal{F}_{\le \mathrm{lcm}(m_1,m_2)}$, any two vertices are connected by a path, which means that there is no homology in degree $0$, and that it is acyclic.
\end{proof}

\begin{Remark}
In the article \cite{dochtermann2012tropical}, the authors produces similar cellular resolutions for the ideal $m^d$. The construction they use comes from a tropical hyperplane arrangement, and their constructions extends the construction made by of Sinefakopoulos in \cite{Si}. Their construction also works for $n\ge 3$, and it should be interesting to investigate if their construction also gives rise to polarizations in general.
\end{Remark}

\section{Polarizations of $(x_1,\dots, x_n)^d_\mathrm{sq.fr}$ and its Alexander dual}\label{sqfr}

When studying polarizations of the ideals $m^d$, it is often easier to study polarizations of the related square-free ideals $I_d=(x_1,\dots, x_n)^d_\mathrm{sq.fr}$. Since these ideals are square-free, they correspond to simplicial complexes, and are in some sense more combinatorial to work with. These ideals also have the property that their Alexander dual $D(I_d)=I_{n-d+1}$ is of the same type. We recall that if $\Delta$ is a simplicial complex on the vertices $[n]$, then the Alexander dual $D(\Delta)=\{F\,|\,[n]\setminus F\not\in \Delta\}$. Equivalently, if $I=I_\Delta$ is a square-free monomial ideal corresponding to a simplicial complex $\Delta$, then $D(I)=I_{D(\Delta)}$ is called the Alexander dual of $I$. If $I$ is generated by the monomials $m_{\sigma_j} = \prod\limits_{i\in \sigma_j} x_i$, it is straight forward to verify that
$$D(I) = \left(n_\tau\,\middle| \, \tau\cap \sigma_j\neq \emptyset \, \forall j\right),$$
where $n_\tau = \prod\limits_{i\in \tau} x_i$. We will see below that if $\widetilde{I}$ is a polarization of $I$, then $D(\widetilde{I})$ is a polarization of $D(I)$. We give a combinatorial description of the polarizations $\widetilde{I_d}$ of $I_d$, which has a natural duality that corresponds to the polarizations $D(\widetilde{I_d})$ of the Alexander duals $D(I_d)$. Finally, we also describe two special maximal polarization which are self dual, in the sense that their Alexander dual are polarizations of the same type. The first is the natural box polarization, and the other one is a natural polarization which actually is standard polarization when $d=2$.

\begin{Theorem}
Let $I$ be a square-free monomial ideal, and let $D(I)$ be its Alexander dual ideal. If $\widetilde{I}$ is a polarization of $I$, then $D(\widetilde{I})$ is a polarization of $D(I)$.
\end{Theorem}

\begin{proof}
First of all, we may without loss of generality assume that the polarization $\widetilde{I}$ is an ideal in $\widetilde{S}=k[x_1, x_{1'}, x_2, x_3, \dots, x_n]$, such that the element $x_1-x_{1'}$ is a non-zero divisor in $\widetilde{S}/\widetilde{I}$, and such that $\widetilde{I}\otimes \widetilde{S}/(x_1-x_{1'}) \cong I$. This is because every polarization is defined to be an iteration of such polarizations, and if the result is true for this case, it must be true in general.

Let $m_{\sigma_1}, m_{\sigma_2},\dots, m_{\sigma_r}$ be a minimal generator set of $I$, where $\sigma_j\subseteq [n]$ and $m_{\sigma_j} = \prod\limits_{i\in \sigma_j} x_i$. We may assume that the sets $\sigma_j$ are chosen such that $\widetilde{I}$ is generated by the monomials $m_{\sigma'_1},\dots, m_{\sigma'_r}$, where $\sigma'_i = \sigma_i$ for $i=1,2,\dots, s$, and $\sigma'_i = \left(\{1'\}\cup \sigma_i\right)\setminus\{1\}$ for $i=s+1,\dots, r$. We can also assume that $1\in \sigma'_i$ for at least one $i$ between $1$ and $s$, since otherwise the polarization would only be a change of variable name.

We now want to show that $D(\widetilde{I})$ in $\widetilde{S}$ is a polarization of $D(I)$. So we need to check that $x_1-x_{1'}$ is a non-zero divisor in $\widetilde{S}/D(\widetilde{I})$, and that $D(\widetilde{I})\otimes \widetilde{S}/(x_1-x_{1'})\cong D(I)$.

First we show that $x_1-x_{1'}$ is a non-zero divisor. So assume otherwise. That is, that there is a square-free monomial $m\in \widetilde{S}$ such that $(x_1-x_{1'})\overline{m} = 0$ in $\widetilde{S}/D(\widetilde{I})$, where $\overline{m}\neq 0$ in $\widetilde{S}/D(\widetilde{I})$. But this means that $x_1m\in D(\widetilde{I})$ and that $x_{1'}m\in D(\widetilde{I})$, while $m\not\in D(\widetilde{I})$. We may assume that $m$ is a square-free monomial and we have that $x_1$ and $x_{1'}$ do not divide $m$. This means that $m = m_\tau$ for some set $\tau\subseteq [n]$ such that $\tau\cap \sigma'_i=\emptyset$ for at least one $i$, while $\left(\{1\}\cup \tau\right) \cap \sigma'_i \neq \emptyset$ and $\left(\{1'\}\cup\tau\right)\cap\sigma'_i \neq \emptyset$ for all $i$. By the construction of $\sigma'_i$, we know that we can not have both $1$ and $1'$ in the same set $\sigma'_i$, so we must have that $\tau\cap\sigma'_a = \emptyset$ for some $a$ between $1$ and $s$ and $\tau\cap\sigma'_b=\emptyset$ for some $b$ between $s+1$ and $r$. But now we can show that $x_1-x_{1'}$ must also be a zero-divisor in $\widetilde{S}/\widetilde{I}$ contradicting the fact that $\widetilde{I}$ is a polarization. To see this, we let $n = \frac{\mathrm{lcm}\left(m_{\sigma'_a}, m_{\sigma'_b}\right)}{x_1x_{1'}}$. We verify that $n\not\in \widetilde{I}$, since if $n\in\widetilde{I}$, then there would have to be a generator $m_{\sigma'_c}=m_{\sigma_c}$ in $\widetilde{I}$ that divides $n$. However, this is not possible, because then we would have that $\left(\{1\}\cup \tau\right)\cap \sigma_c' = \emptyset$ since $\tau\cap \sigma_a = \tau\cap\sigma_b = \emptyset$ and since $1\not\in \sigma_c'$. But this contradicts the fact that $x_1m\in D(\widetilde{I})$ which we assumed earlier. So $n\not\in \widetilde{I}$, while we clearly have that $x_1n\in \widetilde{I}$ and $x_{1'}n\in \widetilde{(I)}$ because of the definition of $n$. This means that $(x_1-x_{1'})\overline{n}=0$ in $\widetilde{S}/\widetilde{I}$ while $\overline{n}\neq 0$. So $x_1-x_{1'}$ is a zero-divisor which is a contradiction. Hence $x_1-x_{1'}$ is a non-zero divisor in $\widetilde{S}/D(\widetilde{I})$ as we wished to prove.

Next, we need to show that in fact $D(\widetilde{I})\otimes \widetilde{S}/(x_1-x_{1'})\cong D(I)$. To do this, we assume that $D(I)$ has a minimal generator set consisting of the monomials $n_{\tau_1},n_{\tau_2},\dots, n_{\tau_p}$. We want to show that $D(\widetilde{I})$ is generated by the monomials $n_{\tau'_1},n_{\tau'_2},\dots, n_{\tau'_p}$ where $\tau'_j = \tau_j$ if $\tau_j\cap \sigma_i \neq \emptyset$ for all $i$, and $\tau'_j = \left(\{1'\}\cup \tau_j\right)\setminus \{1\}$ otherwise.

So suppose that $n_\tau \in D(\widetilde{I})$. We want to show that $n_\tau\in (n_{\tau'_1},n_{\tau'_2},\dots, n_{\tau'_p})$. Since $n_\tau\in D(\widetilde{I})$, we must have that $\tau\cap \sigma'_i \neq \emptyset$ for all $i$. If $\tau\cap \sigma_i\neq \emptyset$ for all $i$, then $n_\tau\in D(I)$, so there is a subset $\tau_j\subseteq \tau$ such that $\tau_j\cap \sigma_i\neq \emptyset$ for all $i$. In this case $\tau'_j = \tau_j$, so we have that $\tau'_j\subseteq \tau$ as well, and $n_\tau\in (n_{\tau'_1},\dots,n_{\tau'_p})$. Next, we assume that $\tau\cap \sigma_i = \emptyset$ for some $i$. This has to mean that $1\not\in \tau$ and $1'\in \tau$. If we define $\rho = \left(\{1\}\cup \tau\right) \setminus\{1'\}$, we must then have that $\rho\cap \sigma_i \neq \emptyset$ for all $i$. So this means that there is a subset $\tau_j\subseteq \rho$ such that $\tau_j\cap \sigma_i \neq \emptyset$ for all $i$. Since we have that $\tau_j\cap \sigma'_i = \emptyset$ for some $i$, we have that $\tau'_j=\left(\{1'\}\cup\tau_j\right)\setminus\{1\}$, and that $\tau'_j\subseteq \tau$. There are now two possibilities. Either $\tau'_j\cap \sigma'_i \neq \emptyset$ for all $i$. In this case we have that $n_\tau\in (n_{\tau'_1},\dots, n_{\tau'_p})$. The other possibility is that $\tau'_j\cap \sigma'_{i'} = \emptyset$ for some $i'$. Since we now both have that $\tau_j\cap \sigma'_i = \emptyset$ for some $i$ and $\tau'_j\cap \sigma'_{i'} = \emptyset$ for some $i'$, we must have that $\tau_j\cap\sigma_a = \tau_j\cap\sigma_{b} = \{1\}$ for some $a$ between $1$ and $s$ and some $b$ between $s+1$ and $r$. But we can exclude this case exactly as we did in the first step above. Because if so, we can show that $x_1-x_{1'}$ is a zero-divisor in $\widetilde{S}/\widetilde{I}$, contradicting the fact that $\widetilde{I}$ is a polarization. To see this, we let $n = \frac{\mathrm{lcm}\left(m_{\sigma'_a},m_{\sigma'_b}\right)}{x_1x_{1'}}$. Again, we see that $n\not\in \widetilde{I}$ since if so, then we must have a generator $m_{\sigma'_c}=m_{\sigma_c}$ dividing $n$. But this is not possible, since then $\tau_j\cap \sigma_c = \emptyset$. It is also obvious that $x_1n\in \widetilde{I}$ and $x_{1'}n\in \widetilde{I}$. So therefore, we have shown that $D(\widetilde{I})$ is generated by the monomials $n_{\tau'_1},\dots,n_{\tau'_p}$, which completes the proof.

\end{proof}

Suppose that $I_d=(x_1,\dots,x_n)^d_\mathrm{sq.fr.}$ and that $\widetilde{I_d}$ is a polarization of $I_d$. We want to describe $\widetilde{I_d}$ in terms of partitions, such that its Alexander dual $D(\widetilde{I_d})$, which is a polarization of $D(I_d)=I_{n-d+1}$, is described by a dual partition. Let
$$\Gamma_d = \{\sigma \in [n]\,|\, |\sigma| = d\}.$$
Suppose that $\widetilde{I_d}$ is generated by the monomials
$$m_\sigma = \prod_{i\in \sigma} x_i^{(a_{i,\sigma})}, \text{ for all }\sigma\in \Gamma_d.$$
Without loss of generality we may assume that the set of all $a_{i,\sigma}$ for a given $i$ is just the set $\{1,\dots,r_i\}$. Define
$$\Sigma_i^d=\{\sigma\in \Gamma_{d-1}\,|\, i\not\in \sigma \}.$$
We will now have a one-to-one correspondence between possible polarizations $\widetilde{I_d}$ of $I$, and partitionings of the sets $\Sigma_i^d$. We say possible polarizations meaning that $\widetilde{I_d}$ modulo the sequence of the differences of variables are isomorphic to $I$, but without the claim that this sequence is a regular sequence. The one-to-one correspondence is as follows. From $\widetilde{I_d}$ as above, we partition
$$\Sigma_i^d=P_{i,1}\cup P_{i,2}\cup \cdots \cup P_{i,r_i}\text{ where }P_{i,j} = \{\sigma\in \Sigma_i^d\,|\, a_{i,\sigma\cup\{i\}}=j\}.$$
And in the other direction, suppose that $\Sigma_i^d = P_{i,1}\cup P_{i,2}\cup \cdots \cup P_{i,r_i}$. Then we let $\widetilde{I_d}$ be the ideal generated by the monomials
$$m_\sigma = \prod_{i\in\sigma} x_i^{(j)}, \text{ where }\sigma\setminus\{i\}\in P_{i,j}.$$

We will illustrate this correspondence with the following example:
\\
\begin{Example}\label{partisjoneringseksempel}
Let for instance $d = 3$ and $n=4$. So we have the ideal $I_3 = (x_1x_2x_3,x_1x_2x_4,x_1x_3x_4,x_2x_3x_4)$. We now have the sets
$$
\begin{array}{ll}
\Sigma_1^3 = & \left\{\{2,3\}, \{2,4\}, \{3,4\}\right\}, \\
\Sigma_2^3 = & \left\{\{1,3\}, \{1,4\}, \{3,4\}\right\}, \\
\Sigma_3^3 = & \left\{\{1,2\}, \{1,4\}, \{2,4\}\right\} \text{ and}  \\
\Sigma_4^3 = & \left\{\{1,2\}, \{1,3\}, \{2,3\}\right\}.
\end{array}
$$

Suppose that $P_{1,1} = \left\{\{2,3\}, \{2,4\}\right\}$, $P_{1,2} = \left\{\{3,4\}\right\}$, $P_{2,1} = \left\{\{1,3\}\right\}, P_{2,2} = \left\{\{1,4\}, \{3,4\}\right\}$, $P_{3,1}=\Sigma_3^3$ and $P_{4,1} = \Sigma_4^3$. We then have the following partitioning
\begin{align*}
\Sigma_1^3 = & P_{1,1}\cup P_{1,2} \\
\Sigma_2^3 = & P_{2,1}\cup P_{2,2} \\
\Sigma_3^3 = & P_{3,1} \\
\Sigma_4^3 = & P_{4,1}.
\end{align*}
This partitioning corresponds to the possible polarization ideal
$$\widetilde{I_d} = \left(x_1^{(1)}x_2^{(1)}x_3^{(1)}, x_1^{(1)}x_2^{(2)}x_4^{(1)}, x_1^{(2)}x_3^{(1)}x_4^{(1)}, x_2^{(2)}x_3^{(1)}x_4^{(1)}\right).$$
To see what the monomial $x_1x_2x_4$ should be in $\widetilde{I_d}$, we see that $x_1$ should go to $x_1^{(1)}$ since $\{2,4\}\in P_{1,1}$, we see that $x_2$ should go to $x_2^{(2)}$ since $\{1,4\}\in P_{2,2}$ and finally that $x_4$ should go to $x_4^{(1)}$ since $\{1,3\}\in P_{4,1}$.

The other way around is easy. For instance, if we have a monomial $x_1^{(2)}x_2^{(3)}x_4^{(1)}$, it means that $\{2,4\}$ should be in a set $P_{1,2}$, that $\{1,4\}$ should be in a set $P_{2,3}$ and that $\{1,2\}$ should be in a set $P_{4,1}$. If we write this out for all generators of $\widetilde{I_d}$ we end up with a partitioning of $\Sigma_1^3,\Sigma_2^3,\Sigma_3^3$ and $\Sigma_4^3$.
\end{Example}

\begin{Def}
Let $\Sigma_i^d=P_{i,1}\cup\cdots\cup P_{i,r_i}$ be a partitioning as described above. We define
$$P_{i,j}^c = \left\{\sigma\,\middle|\, \sigma\subseteq [n],\, |\sigma|=n-d \text{ and } [n]\setminus \left(\sigma\cup\{i\}\right)\in P_{i,j}\right\}.$$
In other words, $P_{i,j}^c$ is the set of the complements of the elements in $P_{i,j}$, where the complements are taken in $[n]\setminus \{i\}$. The partitioning
$$\Sigma_i^{n-d+1} = P_{i,1}^c\cup\cdots\cup P_{i,r_i}^c$$
is called the dual partitioning of $\Sigma_i^d=P_{i,1}\cup\cdots\cup P_{i,r_i}$.
\end{Def}

\begin{Example}
If $\Sigma_1^3,\dots,\Sigma_4^4$ are partitioned as in Example \ref{partisjoneringseksempel} above, then we get the dual partitioning
\begin{align*}
\Sigma_1^2 = & P_{1,1}^c\cup P_{1,2}^c \\
\Sigma_2^2 = & P_{2,1}^c\cup P_{2,2}^c \\
\Sigma_3^2 = & P_{3,1}^c \\
\Sigma_4^2 = & P_{4,1}^c.
\end{align*}
where $P_{1,1}^c = \left\{\{4\}, \{3\}\right\}$, $P_{1,2}^c = \left\{\{2\}\right\}$, $P_{2,1}^c = \left\{\{4\}\right\}$, $P_{2,2}^c = \left\{\{3\}, \{1\}\right\}$, $P_{3,1} = \Sigma_3^2$ and $P_{4,1}^c = \Sigma_4^2$. 
\end{Example}

We can now state the result that shows how polarizations behaves under Alexander duality using this partitioning description.
\begin{Theorem}\label{polarisering}
Suppose that $\widetilde{I_d}$ is a polarization of $I_d$, which corresponds to the partitioning $\Sigma_i^d=P_{i,1}\cup\cdots\cup P_{i,r_i}$. Then the Alexander dual $D(\widetilde{I_d})$ is the polarization of $D(I_d)$ correspondig to the dual partitioning $\Sigma_i^{n-d+1} = P_{i,1}^c\cup \cdots \cup P_{i,r_i}^c$.
\end{Theorem}

\begin{proof}
Let $J$ be the ideal corresponding to the dual partitioning. That is, $J$ is generated by the monomials $n_\tau = \prod\limits_{i\in\tau} x_i^{(j)}, \text{ where }\tau\setminus\{i\}\in P_{i,j}^c$.

First off all, we want to show that $J\subseteq D(\widetilde{I_d})$. So we need to show that if $\tau\subseteq [n]$ is such that $\tau\setminus\{i\}\in P_{i,j}^c$, then we need to show that $n_\tau\in D(\widetilde{I_d})$. That is the same as showing that $\mathrm{Supp}(n_\tau)$ is not a face in $D(\widetilde{\Delta})$, i.e. the Alexander dual of the simplicial complex corresponding to $\widetilde{I_d}$. But the faces in $D(\widetilde{\Delta})$ are just the complements of the faces $\mathrm{Supp}(m_\sigma)$ in $\widetilde{\Delta}$. Thus, it is enough to show that $\mathrm{Supp}(n_\tau)\cap\mathrm{Supp}(m_\sigma)\neq \emptyset$ for all $\tau\in \Gamma_{n-d+1}$ and $\sigma\in \Gamma_d$.

Suppose now that there is a $\tau\in \Gamma_{n-d+1}$ and $\sigma\in \Gamma_{d}$ such that $\mathrm{Supp}(m_\sigma)\cap\mathrm{Supp}(n_\tau) = \emptyset$. Since $|\tau| = n-d+1$ and $|\sigma|=d$, the intersection $\sigma\cap\tau$ must be non-empty. Suppose first that $\sigma\cap\tau=\{i\}$. But then $\tau = (\sigma\setminus\{i\})^c$, and if $\sigma\setminus\{i\}\subseteq P_{i,j}$, then $\tau\setminus\{i\}\subseteq P_{i,j}^c$. But this means that $i^{(j)}:=\mathrm{Supp}(x_i^{(j)})$ is in both $\mathrm{Supp}(m_\sigma)$ and $\mathrm{Supp}(n_\tau)$ which is a contradiction. So $|\sigma\cap\tau|=k>1$. Without loss of generality, we may assume that $\sigma\cap\tau = \{1,\dots, k\}$. It is clear that there are sets $\sigma_1,\dots,\sigma_k$ such that $\tau=(\sigma_i\setminus\{i\})^c$ for all $i\in \{1,\dots, k\}$. Let $i^{(a_i)}$ be the polarized vertex of $i$ in $\mathrm{Supp}(m_{\sigma_i})$, and $i^{(s)}$ the polarized vertex of $i$ in $\mathrm{Supp}(m_\sigma)$. If $a_i=s$ for some $i$, we will get a contradiction as in the case above. So we may assume that $i^{(a_i)}\neq i^{(s)}$ for all $i\in \{1,\dots,k\}$. We will now show that either there exists a $\sigma'$ such that $\mathrm{Supp}(n_\tau)\cap\mathrm{Supp}(m_{\sigma'}) = \emptyset$ and $|\sigma'\cap \tau|<k$, or that the elements $x_i^{(a_i)}-x_i^{(s)}$ are zero divisors in $\widetilde{S}/\widetilde{I_d}$. In any case, iterating the argument if neccessary, we will either end up with a zero divisor for $\widetilde{S}/\widetilde{I_d}$, which contradicts the fact that $\widetilde{I_d}$ is a polarization of $I_d$, or we end up with a $\sigma''$ such that $|\sigma''\cap\tau|=1$ giving the contradiction above. Let
$$f=\frac{\mathrm{lcm}(m_{\sigma_1}, m_\sigma)}{x_1^{(a_1)}x_1^{(s)}}.$$
Then it is clear that $(x_1^{(a_1)}-x_1^{(s)})\overline{f}=0$ so either $x_1^{(a_1)}-x_1^{(s)}$ is a zero divisor in $\widetilde{S}/\widetilde{I_d}$, or $\overline{f}=0$ in $\widetilde{S}/\widetilde{I_d}$. But if $\overline{f}=0$, it means that there is a $\sigma'$, such that $m_{\sigma'}$ divides $f$. But by the construction of $f$, we must have that $\sigma'\cap \tau=\alpha$, with $\alpha\subseteq \{2,\dots,k\}$. But now $\mathrm{Supp}(m_{\sigma'})\cap \mathrm{Supp}(n_\tau)=\emptyset$, since we have assumed that $\mathrm{Supp}(m_{\sigma})\cap \mathrm{Supp}(n_\tau)=\emptyset$, and since $i^{(s)}\in \mathrm{Supp}(m_{\sigma'})$ is also $i^{(s)}$ in $\mathrm{Supp}(m_\sigma)$, for $i\in \alpha$. This is because of the definition of $f$. Since $|\alpha|\le k-1$, we are done.

Next, we will have to show that $D(\widetilde{I_d})\subseteq J$. But this is clear since we know from Theorem \ref{polarisering} that $D(\widetilde{I_d})$ is a polarization of $D(I_d)$. This means that $D(\widetilde{I_d})$ and $D(I_d)$ should have the same number of generators, and they are all of degree $n-d+1$. Since $J$ are already generated by this number of generators of degree $n-d+1$, it is clear that $D(\widetilde{I_d})\subseteq J$, and it follows that $J=D(\widetilde{I_d})$. 
\end{proof}

It is possible to give a description of when a partitioning of the sets $\Sigma_i^d$ correspond to a polarization of the ideal $I_d$. It is straight forward, but in general it is difficult to use.

\begin{Prop}\label{partisjonering}
A partitioning $P$ corresponds to a polarization of $I_d$ if and only if the following holds: For any $\sigma\in P_{i,j}$ and $\tau\in P_{i,j'}$, $j\neq j'$, then there exists a $\beta\subseteq \sigma\cup\tau$, such that $|\beta|=d$ and for all $t\in \beta$ then if $\beta\setminus\{t\}\in P_{t,s}$, then $\sigma\cup\{i\}\setminus\{t\}\in P_{t,s}$ or $\tau\cup\{i\}\setminus\{t\}\in P_{t,s}$, or both.
\end{Prop}

\begin{proof}
The partitioning $P$ gives an ideal $\widetilde{I_d}$ which is a polarization of $I_d$ if and only if $x_i^{(j)}-x_i^{(j')}$ is a non-zero divisor in $\widetilde{S}/\widetilde{I_d}$, for all $i,j,j'$. Any annihilator for $x_i^{(j)}$ is divisible by $\frac{m_{\sigma\cup\{i\}}}{x_i^{(j)}}$, for a $\sigma\in P_{i,j}$, and an annihilator for $x_i^{(j')}$ is divisible $\frac{m_{\tau\cup\{i\}}}{x_i^{(j')}}$, for a $\tau\in P_{i,j'}$. So any annihilator of $x_i^{(j)}-x_i^{(j')}$ is divisible by $m=\mathrm{lcm}\left(\frac{m_{\sigma\cup\{i\}}}{x_i^{(j)}}, \frac{m_{\tau\cup\{i\}}}{x_i^{(j')}}\right)$. But $\overline{m}=0$ in $S/\widetilde{I_d}$ if and only if $m_\beta$ divides $m$ for a $\beta\in \Gamma_d$. But $m_\beta$ divides $m$ if and only if $\beta\setminus\{t\}$ is in the same partition as $\sigma\cup\{i\}\setminus\{t\}$ or $\sigma\cup\{i\}\setminus\{t\}$ in the partitioning of $\Sigma_t^d$, for all $t\in\beta$.
\end{proof}

\begin{Example}
The box polarization corresponds to the following partitioning. If $\sigma\in \Sigma_i^d$, and $\sigma = \{s_1,s_2,\dots, s_{d-1}\}$, such that $s_1<s_2<\cdots < s_{r-1} < i < s_{r}<\cdots < s_{d-1}$, then $\sigma\in P_{i,r}$. Although it is well known that the box polarization is a polarization, we can use Proposition \ref{partisjonering} to verify this fact. So let $\sigma\in P_{i,j}$ and $\tau\in P_{i,j'}$, and we may assume that $j<j'$. Then we choose $\beta\subseteq \sigma\cup\tau$ in the following way. For $a<j$, we choose $\beta_a = \sigma_a$, for $j\le a < j'$ we choose $\beta_a = \tau_a$, and finally for $j'\le a$ we choose $\beta_a = \sigma_a$. Now it is straight forward to verify that $\beta\setminus\{t\}$ is in the same partition as $\sigma\cup\{i\}\setminus\{t\}$ or $\tau\cup\{i\}\setminus\{t\}$ depending on the position of $t$ in $\beta$. Furthermore, it is easy to verify that the box polarization is a maximal polarization. For suppose that $P'$ is a partitioning that refines $P$. This means that for some $i$ there exists $\sigma$ and $\tau$ in different partitions of $\Sigma_i^d$, but where $i$ would have position $r$ in both. Now let $\beta\subseteq \sigma\cup \tau$ be any subset such that $|\beta|=d$. If we consider the $t$ in $\beta$ in position $r$, it follows from the box partitioning that $\beta\setminus\{t\}\in P_{t,r}$, but the $t$ can not have position $r$ in $\sigma\cup\{i\}$ nor in $\tau\cup\{i\}$. Therefore the condition of Proposition \ref{partisjonering} does not hold, and $P'$ is not a polarization.

We can show that the Alexander dual of this polarization is in fact of the same type, by observing that if $\sigma\in P_{i,r}$ and if $\tau\in P_{i,r}^c$ is its complement. Then $\tau = \{t_1,t_2,\dots, t_{n-d}\}$ where $t_1>t_2>\dots,t_{r-1}>i>t_r>\dots t_{n-d}$.

The box polarization of $(x_1,\dots,x_n)^d_\mathrm{sq.fr.}$ is always the box polarization of $(x_1,\dots,x_{n'})^d$, where $n'=n-d+1$. See \cite[Theorem 3.13 iii)]{NR} for more details.

\end{Example}

\bigskip

\begin{Example}
Another example is the one where we polarize one variable to as many variables as possible. For a fixed $i$, if we have $\Sigma_i^d = \{\sigma_1,\sigma_2,\dots,\sigma_r\}$, we can now form the partitioning where $P_{i,j} = \{\sigma_j\}$, and $P_{i',1}=\Sigma_{i'}$ for all $i'\neq i$. By Proposition \ref{partisjonering} this is clearly a polarization. It is also easy to show that this polarization is maximal. If $P'$ is a partitioning that refines $P$, then it is possible to find a $\sigma$ and $\tau$ in $\Sigma_j^d$, such that $\sigma\in P'_{j,s}$, $\tau\in P'_{j,s'}$, $i\in \sigma$ and $|\sigma\cup\tau|=d$. We show that $P'$ can not be a polarization by using Proposition \ref{partisjonering}. Since $|\sigma\cup\tau|=d$, then we only need to check the criteria for $\beta=\sigma\cup\tau$. But the criteria fails since $\beta\setminus\{i\}$ and $\sigma\cup\{j\}\setminus\{i\}$ must be in different partitions of $\Sigma_i^d$ since the partitioning has only one element in each part. So $P'$ is not a polarization.

By the construction of the dual partition it is easy to see that this construction is self dual, meaning that the Alexander dual is constructed by a partitioning of exactly the same type. 

For $d=2$, this polarization of $(x_1,\dots,x_n)^d_\mathrm{sq.fr.}$ is the standard polarization of $(x_1,\dots,x_{n-1})^2$. However, this is not true for $d>2$ as this polarization has more variables than the standard polarization of $(x_1,\dots,x_{n'})^d$. See Example \ref{standardpolar} below.

\end{Example}

\section{Polarizations of $(x_1,\dots,x_n)^2_\mathrm{sq. fr.}$ and its Alexander dual}\label{sqfr2}
Let $I=(x_1,\dots,x_n)^2_\mathrm{sq. fr.}$, and let $D(I)=(x_1,\dots,x_n)^{n-1}_{\mathrm{sq.fr.}}$. As we have seen in the previous section, polarization of $I$ and $D(I)$ are dual to each other, so we will study polarization of these ideals simultaneously. We want to study different polarizations of the ideal $I$. We will show that polarizations of $I$ and $D(I)$ naturally corresponds to connected subgraphs of the complete graph $K_n$ on $n$ vertices. However, not all connected subgraphs correspond to polarizations of $I$ and $D(I)$. We show a one-to-one correspondence between maximal polarizations and spanning trees of $K_n$, and we show that two special spanning trees correspond to the box polarization and the standard polarization.

\begin{Lemma}
If $J$ is a polarization of $D(I)$ then there exists a path of linear relations between any two generators.
\end{Lemma}
\begin{proof}
If $J$ is a polarization of $D(I)$, then $J$ is also linear of codimension $2$. So it has a cellular linear minimal free resolution which must be supported on a tree consisting of linear relations between its generators. See \cite{Fl} for more details.
\end{proof}

\begin{Lemma}\label{AD}
If $J$ is a polarization of the ideal $D(I)$, and if $J$ is generated by the monomials
$$m_i=x_1^{(a_{i,1})}\cdots \widehat{x_i} \cdots x_n^{(a_{i,n})},$$
then the Alexander dual of $J$ is the ideal generated by the monomials $x_i^{(a_{j,i})}x_j^{(a_{i,j})}$, for $i\neq j$, and $D(J)$ is a polarization of $I$.
\end{Lemma}

\begin{proof}
This follows from Theorem \ref{polarisering}. For instance, if $i<j$, then in the partition corresponding to $J$ we must have that $\{1,2,\dots,\widehat{i},\dots,\widehat{j}\dots, n\}\in P_{i,a_{i,j}}$. This means that $\{j\}\in P_{i,a_{i,j}}^c$, which means that $x_j$ is polarized to $x_j^{(a_{i,j})}$ in the monomial $x_ix_j$. Using exactly the same argument we can show that $x_i$ is polarized to $x_i^{(a_{j,i})}$ in the monomial $x_ix_j$, and the result follows.
\end{proof}

So if $J$ is a polarization of $D(I)$, and if we let the generators of $J$ correspond to the vertices in the graph $K_n$, then the linear relations between the generators correspond to a connected subgraph of $K_n$. On the other hand, we want to show that any spanning tree for $K_n$ corresponds to a maximal polarization of $D(I)$. The technique we will use is a square-free version of the technique used in \cite{Fl}.

Let $T$ be a spanning tree for $K_n$, and label the edges in $T$ by $e_1,\dots,e_{n-1}$. Every edge $e_i=(v_i, w_i)$ disconnects $T$ into two parts, and we label the vertices in the component containing $v_i$ by $w_i^{(i)}$ and the vertices in the component containing $w_i$ by $v_i^{(i)}$. For each vertex $v$ in $T$, we now get a monomial 
$$m_v:=\prod_{v\text{ is labeled by }j^{(i)}} x_j^{(i)},$$
and we let $P=(m_1,\dots,m_n).$

\begin{Prop}
The ideal $P$ is a polarization of $D(I)$ and the linear relations between the generators are precisely the edges given by the tree $T$.
\end{Prop}

\begin{proof}
By the construction of $P$, we see that a linear relation between $m_i$ and $m_j$ occurs if and only if the labels in the vertices $i$ and $j$ differs by only one, that is if and only if $e=(i,j)$ is an edge if $T$.
\end{proof}

We can use Lemma \ref{AD} to give an explicit description of the Alexander dual of the ideal above.
%So for any tree we can describe the polarization of $(x_1,\dots,x_n)^2_{\mathrm{sq.fr.}}$

\begin{Prop}\label{PolAD}
Let $T$ be a spanning tree for the complete graph on $\{1,\dots,n\}$, and suppose the edges in $T$ are labeled by $e_1,\dots,e_{n-1}$, and let $P$ be the ideal described above. Then $D(P)$ is generated by the monomials $x_i^{(p)}x_j^{(q)}$, where $i\neq j$ and the unique path from $i$ to $j$ in $T$ starts in $e_p$ and ends in $e_q$.
\end{Prop}

\begin{proof}
This follows immediately from Lemma \ref{AD}. We just need to verify that if $i$ and $j$ are two vertices in $K_n$, then $x_j^{(q)}$ divides $m_i$ and $x_i^{(p)}$ divides $m_j$ (i.e. $a_{i,j}=q$ and $a_{j,i}=p$). This is just the same as to say that the vertex $i$ is labeled by $j^{(q)}$ and that $j$ is labeled by $i^{(p)}$. By the construction this is clear when the unique path from $i$ to $j$ starts in $e_p$ and ends in $e_q$. 

\end{proof}

\begin{Example}\label{boxpolar}
The box polarization comes from the following tree:

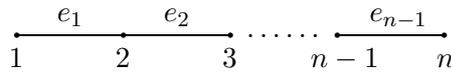
\begin{center}
\begin{picture}(200,30)
\put(10,12){\circle*{2}}
\put(7,0){$1$}
\put(25,17){$e_1$}
\put(10,12){\line(1,0){40}}
\put(50,12){\circle*{2}}
\put(47,0){$2$}
\put(65,17){$e_2$}
\put(50,12){\line(1,0){40}}
\put(90,12){\circle*{2}}
\put(87,0){$3$}
\put(96,11){$\dots \dots$}
\put(130,12){\circle*{2}}
\put(120,0){$n-1$}
\put(142,17){$e_{n-1}$}
\put(130,12){\line(1,0){40}}
\put(170,12){\circle*{2}}
\put(167,0){$n$}

\end{picture}
\end{center}

This is because from Proposition \ref{PolAD}, we have that $D(P)$ is generated by the monomials $x_i^{(i)}x_j^{(j-1)}$, $1\le i < j \le n$, but if we identify each variable $x_i^{(i)}$ with $x_i^{(1)}$, and each variable $x_j^{(j-1)}$ with $x_{j-1}^{(2)}$, we see that $D(P)$ is generated by the monomials $x_i^{(1)}x_{j-1}^{(2)}$ for $1\le i<j\le n$, but this is just the same as the ideal generated by the monomials $x_i^{(1)}x_{j'}^{(2)}$ for $1\le i\le j' \le n-1$, which is the box polarization of $(x_1,\dots,x_{n-1})^2$.

\end{Example}

\bigskip
\begin{Example}\label{standardpolar}
The standard polarization comes from the following tree:
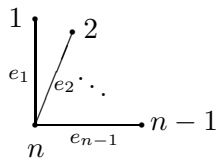
\begin{center}
\begin{picture}(80,70)
\put(10,12){\circle*{2}}
\put(7,0){$n$}
\put(10,12){\line(0,1){40}}
\put(0,31){$_{e_1}$}

\put(10,52){\circle*{2}}
\put(0,49){$1$}

\put(10,12){\line(2,5){14}}
\put(17,27){$_{e_2}$}
\put(24,47){\circle*{2}}
\put(28,45){$2$}

\put(10,12){\line(1,0){40}}
\put(23,7){$_{e_{n-1}}$}

\put(53,10){$n-1$}
\put(50,12){\circle*{2}}

\put(25,22){$\ddots$}

\end{picture}
\end{center}

Again this is because of Proposition \ref{PolAD}. We have that $D(P)$ is generated by the monomials $x_i^{(i)}x_n^{(i)}$ for $1\le i < n$ and $x_i^{(i)}x_j^{(j)}$ for $1\le i < j < n$. Now if we identify the variables $x_i^{(i)}$ by $x_i^{(1)}$ and $x_n^{(j)}$ by $x_j^{(2)}$, we have that $D(P)$ is generated by the monomials $x_i^{(1)}x_i^{(2)}$ for $1\le i < n$ and $x_i^{(1)}x_j^{(1)}$ for $1\le i < j < n$, which is the standard polarization of $(x_1,\dots, x_{n-1})^2$.
\end{Example}

\begin{Prop}
Let $J$ be a maximal polarization of the ideal $(x_1,\dots,x_n)^2$. Then $J$ is isomorphic to a maximal polarization of the ideal $I_2 = m_{\mathrm{sq.fr}}^2$.
\end{Prop}

\begin{proof}
Let $J$ be a maximal polarization of the ideal $(x_1,\dots,x_n)^2$. This means that we may assume that $J$ is a square-free monomial ideal in the polynomial ring $k[x_1^{(1)},x_1^{(2)},x_2^{(1)},x_2^{(2)},\dots,x_n^{(1)},x_n^{(2)}]$. If we take the Alexander dual of $J$, we now get a square-free monomial ideal $D(J)$ in the same polynomial ring. Since $J$ is Cohen--Macaulay, and have a linear resolution, we can use the Eagon--Reiner theorem to conclude that $D(J)$ is Cohen--Macaulay of codimension two. So by \cite[Proposition 2.1]{Fl}, there is a labeled tree which gives a cellular resolution of $D(J)$. Furthermore, by \cite[Theorem 2.4]{Fl}, we can use the labelling technique above to produce a unique maximal labelling of this tree, up to isomorphism. If $J$ is a maximal polarization, then $D(J)$ is a maximal polarization, so it must be isomorphic to the ideal we get from this tree. Since this ideal really is a polarization of the ideal $I_{n+1}$, we use Theorem \ref{polarisering} to show that $D(D(J)) = J$ is isomorphic to a polarization of the ideal $I_d$. 
\end{proof}
\mbox{} 

\begin{Remark}
For $n\ge 3$ and $d\ge 3$ this result are no longer true. For instance, it is possible to show that the standard polarization of $(x_1,x_2,x_3)^3$ does not come from a polarization of the ideal $(x_1,\dots,x_5)_\mathrm{sq.fr.}^3$. This can be shown by counting the number of the different variables occuring in the standard polarization. In this case, the variables $x_1^{(1)}$, $x_2^{(1)}$ and $x_3^{(1)}$ occur each $6$ times, while the variables $x_1^{(2)}$, $x_2^{(2)}$ and $x_3^{(2)}$ occur $3$ times, and the variables $x_1^{(3)}$, $x_2^{(3)}$ and $x_3^{(3)}$ occur only once. If this now is a polarization of the ideal $(x_1,\dots,x_5)_\mathrm{sq.fr.}^3$, it must correspond to a partitioning of the set of sets $\Sigma_1^3,\dots, \Sigma_5^3$. Since $\Sigma_i^3$ consists of $6$ sets each, we must have a partitioning like for instance $\Sigma_1^3 = P_{1,1}$, $\Sigma_2^3=P_{2,1}$, $\Sigma_3^3 = P_{3,1}$, $\Sigma_4^3 = P_{4,1}\cup P_{4,2}$, and $\Sigma_5^3=P_{5,1}\cup P_{5,2}\cup P_{5,3}\cup P_{5,4}$, where $|P_{1,1}|=|P_{2,1}|=|P_{3,1}|=6$, $|P_{4,1}|=|P_{4,2}|=|P_{5,1}|=3$ and $|P_{5,2}|=|P_{5,3}|=|P_{5,4}|=1$. To see that this uneven partitioning of the sets can not correspond to the standard polarization, we need to do a case by case study of possible partitionings of this form. The details are not included in in this paper. 
\end{Remark}
\mbox{} 

\begin{Remark}
We showed in this section that maximal polarizations of $I_2$ corresponded to spanning trees of the complete graph. However, there exists many other non-maximal polarizations of this ideal. They correspond to partitioning of the sets $\Sigma_i^2=\textrm{Supp}(m_i) = \{1,\dots,\widehat{i}\dots, n\}$. In this case the result of Proposition \ref{partisjonering} can be given more explicitely.
\end{Remark}

\begin{Prop} A partitioning $\Sigma_i^2 = P_{i,1}\cup\cdots\cup P_{i,r_i}$ corresponds to a polarization of $I_2$ if and only if the following condition is satisfied:
$$\text{If }j\not\in P_{i,s}, \text{ then } \{i\}\cup P_{i,s}\subseteq P_{j,t} \text{ for some } t.$$
\end{Prop}

\begin{proof}
If $j\not\in P_{i,s}$ it means that we can choose $j\in P_{i,s'}$ and $j'\in P_{i,s}$. By Proposition \ref{partisjonering} we can then create $\beta = \{j,j'\}$, and we must have that $\{j'\}=\beta\setminus\{j\}$ and $\{i\}=\{j\}\cup\{i\}\setminus\{j\}$ are in the same partition of $\Sigma_j^d$. Since we get this claim for all $j'\in P_{i,s}$, we must have that $\{i\}\cup P_{i,s}\subseteq P_{j,t}$ for some $t$.
\end{proof}

This condition is technical and not so easy to visualize combinatorial. We may however give a more geometrical approach to explain how these polarizations behave.

\begin{Def}
Let $G$ be a graph with vertices $1,2,\dots,n$. The ideal
$$I_G = \left(x_ix_j\,\middle| \, (i,j)\in G\right)$$
in $k[x_1,\dots,x_n]$ is called the edge ideal of $G$.
\end{Def}

\begin{Example}
The ideal $I = (x_1,\dots,x_n)_{\mathrm{sq.fr}}^2$ is the edge ideal of $K_n$. Furthermore, if the ideal $\widetilde{I}$ is a polarization of $I$ where one variable $x_i$ is polarized into the variables $x_i$ and $x_{i'}$, then $\widetilde{I}$ is the edge ideal of the graph obtained by polarizing the vertex $i$ into $i$ and $i'$. 
\end{Example}

Since every polarization $\widetilde{I}$ of $I$ can be obtained inductively by successive polarization like this, we need a criterion for whenever an edge ideal can be polarized. We can then use this criteria for examining all possible polarization of the ideal $I_2$.

\begin{Prop}
Let $G$ be a graph on the vertices $1,\dots,n$, and let $I_G$ be the corresponding edge ideal. Then there exists a polarization $\widetilde{I_G}$ of $I_G$, where the variable $x_i$ is polarized into $x_i$ and $x_{i'}$ if and only if the graph $G|_{\mathrm{link}(i)}$, i.e. the graph $G$ restricted to the vertices which are neighbour of $i$, contains a complete bipartite graph on all its vertices. Every choice of such a bipartitioning corresponds to a polarization of $G$ by letting $i$ be neighbour to the vertices in one part and $i'$ be neighbour to the other part, and $\widetilde{I_G}$ is made similarly by choosing $x_i$ in the monomials corresponding to one part, and $x_{i'}$ in the other part. 
\end{Prop}

\begin{proof}
Assume that $\widetilde{I_G}$ is a possible polarization of $I_G$, where the variable $x_i$ is polarized into two variables $x_i$ and $x_{i'}$. This corresponds to a bipartitioning of the neighbourhood vertices $V=A\cup B$. The ideal is a polarization if and only if $x_i-x_{i'}$ is a non-zero divisor in $\widetilde{S}/\widetilde{I_G}$. So assume that there is a monomial $m$ such that $(x_i-x_{i'})\overline{m} = 0$. But this means that $x_i\overline{m} = x_{i'}\overline{m}=0$. This means that $x_im\in \widetilde{I_G}$ and $x_{i'}m\in \widetilde{I_G}$. This happens if $m$ is divisible by any variable $x_a$ with $a\in A$ and any variable $x_b$ with $b\in B$. So $x_i-x_{i'}$ is a zero-divisor if there exists a monomial $m = x_ax_b$ such that $a\in A$, $b\in B$ and $x_ax_b\not\in \widetilde{I_G}$. But this happens if and only if $G|_V$ does not contain a complete bipartite graph $K_{A,B}$.
\end{proof}

%\begin{Example}

\bibliography{polar}

\end{document}